\documentclass[12pt]{amsart}
\usepackage{amscd}
\usepackage{amssymb}
\usepackage{tikz-cd}
\usepackage{hyperref}
\usepackage{enumitem}

\newtheorem{Theorem}{Theorem}[section]
\newtheorem{Lemma}[Theorem]{Lemma}
\newtheorem{Corollary}[Theorem]{Corollary}
\newtheorem{Proposition}[Theorem]{Proposition}

\theoremstyle{remark}
\newtheorem{Example}[Theorem]{Example}
\newtheorem{Remark}[Theorem]{Remark}

\DeclareMathOperator{\ar}{ar}
\DeclareMathOperator{\gr}{gr}

\def\To {\longrightarrow}
\def\sk{\smallskip\par}
\def\mm{{\frak m}}

\def\gr{{\mathrm{gr}}}
\def\ar{{\mathrm{ar}}}
\def\inn{{\mathrm{in}}}

\newcommand{\Jbar}{\overline J}
\newcommand{\Rbar}{\overline R}
\newcommand{\QNVT}{\mathrm{QNVT}}

\def\Ker{\operatorname{Ker}}
\def\Ext{\operatorname{Ext}}
\def\Hom{\operatorname{Hom}}

\def\Im{\operatorname{Im}}

\def\Ext{\operatorname{Ext}}

\def\max{\operatorname{max}}

\begin{document}
\title{Equivalence of Extensions and the Perturbation Index}
\author{Ton That Quoc Tan}
\address{Department of Mathematics, FPT University, Danang, Viet Nam}
\email{tanttq@fe.edu.vn, quoctanmath@gmail.com}
\subjclass[2020]{Primary 13A30, 13D45; Secondary 13H15, 14B12}
\keywords{small perturbation, filter regular sequence, initial ideal, associated graded ring, Artin-Rees number}
\maketitle

\begin{abstract}
        In this paper, using the equivalence of two extensions, we obtain an improved upper bound for the perturbation index, refining the bound of Quy and N. V. Trung.
\end{abstract}
\section{Introduction}
        Let $(R,\mm)$ be a Noetherian local ring, $J$ be an ideal of $R$ and
        $I=(f_1,\ldots,f_r)$. For a positive integer $N$, we consider ideals of the form
        $$
        I'=(f_1+\varepsilon_1,\ldots,f_r+\varepsilon_r),
        $$
        where $\varepsilon_i\in J^N$ for all $i=1,\ldots,r$. Such an ideal $I'$ is called
        a \emph{$J$-adic perturbation} of $I$. When $J=\mm$, it is also called a
        \emph{small perturbation} of $I$.

        Perturbation problems arise naturally in singularity theory. 
        For instance, when studying the singularities of an analytic space we often want to replace the defining power series by their $n$-jets for some $n \gg 0$. This problem has been studied in singularity theory, mainly with the aim to find necessary or sufficient conditions for $R/I \cong R/I'$ (see e.g. \cite{CS1,CS2,GP,MS}). \par

        The main problem considered in this paper concerns the stability of the associated
        graded ring
        $$
        \gr_J(R/I)=\bigoplus_{n\geq 0}\frac{J^n+I}{J^{n+1}+I}
        $$
        under $J$-adic perturbations of $I$. More precisely, we are interested in finding
        an integer $N$ such that
        $
        \gr_J(R/I')\cong\gr_J(R/I)
        $
        whenever $\varepsilon_1,\ldots,\varepsilon_r\in J^N$. The smallest such integer
        is called the \emph{perturbation index} of $\gr_J(R/I)$.

        In 1996, Srinivas and Trivedi \cite{ST} showed that the Hilbert--Samuel function
        of $R/I$ with respect to an $\mm$-primary ideal $J$ is stable under small
        perturbations when $R$ is a generalized Cohen--Macaulay ring and
        $f_1,\ldots,f_r$ is part of a system of parameters. They conjectured that the same
        conclusion holds for an arbitrary local ring provided that $f_1,\ldots,f_r$ is a
        filter-regular sequence. This conjecture was later solved by Ma, Quy and Smirnov
        \cite{MQS}. In fact, the results of Srinivas and Trivedi \cite{ST} and Ma, Quy and Smirnov
        \cite{MQS} establish the stronger stability
        $
        \gr_J(R/I)\cong\gr_J(R/I'),
        $
        from which the stability of the Hilbert--Samuel function follows. The corresponding
        problem for Artinian modules and coregular sequences was studied by Tan
        \cite{TanBKMS}. Other invariants under perturbations have also been investigated.
        Duarte \cite{Duarte} studied the stability of Betti numbers, while V. D. Trung
        \cite{VDT} considered both Bass and Betti numbers. D. T. Cuong, L. T. Nhan and
        N. X. Linh \cite{CNL} studied E-depth, Ef-depth and the dimension of the
        non-sequentially Cohen--Macaulay locus under perturbations of ideals.
        Recently, N. H. Duc, N. D. Hop and P. H. Quy \cite{DHQ} studied how the properties of being reduced, an integral domain, and normal, behave under small perturbations. 

        By the Artin-Rees Lemma, for every ideal $I$ of $R$, there exists an integer $c$ such that
        $$
        J^n\cap I=J^{n-c}(J^c\cap I)
        $$
        for every $n\geq c$. The least such integer $c$ is called the
        \emph{Artin--Rees number} of $I$ with respect to $J$, and is denoted by
        $\ar_J(I)$.

For an $R$-module $M$, the \emph{$J$-Loewy length} of $M$ is defined by
$$
a_J(M):=\inf\{n \geq 0 \mid J^nM=0\}.
$$
For an arbitrary ideal $J$, Quy and N. V. Trung \cite{QNVT} obtained an explicit
upper bound for the perturbation index when $f_1,\ldots,f_r$ is a $J$-filter
regular sequence. For $i=1,\ldots,r$, let
$$
a_i=
a_J\left(
\frac{(f_1,\ldots,f_{i-1}):f_i}
     {(f_1,\ldots,f_{i-1})}
\right);
\qquad 
$$
their bound is 
$$
N_{\rm QNVT}=\max\{a_1+2a_2+\cdots+2^{r-1}a_r,\ar_J(f_1),\ldots,\ar_J(f_1,\ldots,f_r)\}+1.
$$
More recently, Tan \cite{Tanarxiv} improved this bound to
$$
N_{\rm Tan}
=
\max\{a_1+a_2+2a_3+\cdots+2^{r-2}a_r,\ar_J(f_1),\ldots,\ar_J(f_1,\ldots,f_r)\}+1.
$$
Thus a natural problem is whether the exponential coefficients appearing in these
bounds can be reduced.

The purpose of this paper is to obtain a further refinement of the perturbation
bound of Quy and N. V. Trung. Our method is based on the equivalence of two
extensions. More precisely, we construct explicit $1$-cocycles representing two extensions
arising before and after perturbing the first element and show that these
$1$-cocycles coincide. This gives a refinement of
\cite[Proposition 3.4]{QNVT} and shows that the associated $J$-Loewy lengths are preserved under perturbation.

Our main result is the following theorem.

\begin{Theorem}[Main theorem]\label{introductionthm}
Let $(R,\mm)$ be a Noetherian local ring and $J$ an arbitrary ideal of
$R$. Let $f_1,\ldots,f_r$ be a $J$-filter regular sequence on $R$ and $I=(f_1,\ldots,f_r)$.
Let
$$
a_i=
a_J\left(
\frac{(f_1,\ldots,f_{i-1}):f_i}
     {(f_1,\ldots,f_{i-1})}
\right)
$$
for $i=1,\ldots,r$. Set
$
T= \max_{2\le i\le r} \left\{ a_i-\sum_{j=i+1}^r a_j \right\} \quad \text{and} \quad T=0 \text{ if } r=1,
$
$$
N=
\max\left\{
a_1+\cdots+a_r+T,
\ar_J(f_1),
\ldots,
\ar_J(f_1,\ldots,f_r)
\right\}+1.
$$
Then, for every
$
I'=(f_1',\ldots,f_r'),
$
where
$
f_i'=f_i+\varepsilon_i,$
with $
\varepsilon_i\in J^N,
$
for $i=1,\ldots,r$, we have

{\rm (i)} $f_1',\ldots,f_r'$ is a $J$-filter regular sequence with
$$
a_J\left(
\frac{(f_1',\ldots,f_{i-1}'):f_i'}
     {(f_1',\ldots,f_{i-1}')}
\right)
=
a_i
$$
for $i=1,\ldots,r$.

{\rm (ii)}
$
\inn(f_1',\ldots,f_i')
=
\inn(f_1,\ldots,f_i)
$
for $i=1,\ldots,r$.

{\rm (iii)} In particular,
$
\inn(I')=\inn(I), 
\gr_J(R/I')\cong\gr_J(R/I)$
and $
\ar_J(I')=\ar_J(I).
$
\end{Theorem}
If $r \geq 2$ then
$$N=\max\{a_1+2a_2,a_1+a_2+2a_3,\ldots,a_1+\cdots+a_{r-1}+2a_r,\ar_J(f_1),\ldots,\ar_J(f_1,\ldots,f_r)\}+1.$$
The bound $N$ is always no larger than the bound $N_{\rm QNVT}$, i. e.
$
N \leq N_{\rm QNVT}.
$
Thus the main theorem gives a refinement of the perturbation bound in
\cite{QNVT}.
For example, if $r\ge2$, $a_1=\dots=a_r=1$ and $\mathrm{ar}_J(f_1,\dots,f_i)=0$ for all $i$,
then $N_{\mathrm{QNVT}}=2^r$, $N_{\mathrm{Tan}}=2^{r-1}+1$ and $N=r+2$. Thus $N$ grows
linearly in $r$ and $N<N_{\mathrm{Tan}}$ as soon as $r\ge4$ ($N$ and $N_{\mathrm{Tan}}$ are not comparable in
general, see Remark \ref{Remark}).

The paper is organized as follows. Section~2 recalls the basic facts from \cite{QNVT}.
In Section~3, we study the equivalence of the extensions arising before and after
perturbing the first element and prove Theorem~\ref{introductionthm}.
Section \ref{application} give an application to Bass and Betti numbers.
\section{Preliminaries}

Throughout this paper $(R,\mm)$ is a Noetherian local ring and $J$ is an arbitrary ideal of $R$.  For a finitely generated $R$-module $M$, 
we use the convention $a_J(M)=\infty$ if $J^nM \neq 0$ for all $ n\geq 0$.
 
A sequence $f_1,\ldots,f_r$ in $R$ is called \emph{$J$-filter regular} if for every $i =1, \ldots, r$
$$
 a_J\!\left(
 \frac{(f_1,\ldots,f_{i-1}):f_i}{(f_1,\ldots,f_{i-1})}
 \right)<\infty.
$$

For an ideal $I\subseteq R$, let $\inn(I)$ denote \emph{initial ideal} of $I$ in $\gr_J(R)$.  We use the convention of \cite{QNVT}; in particular,
\begin{equation}\label{isomorphism}
  \gr_J(R/I)\cong \gr_J(R)/\inn(I).
\end{equation}
By \cite[Proposition 2.3]{QNVT},
\begin{equation}\label{AR-initial}
 \ar_J(I)=d(\inn(I)),
\end{equation}
where $d(\inn(I))$ is the maximal degree of a minimal homogeneous generating set of $\inn(I)$.

We collect the results from \cite{QNVT} used below.

\begin{Lemma}\label{artinreesthuong}
Let $K\subseteq I$ and put $\Rbar=R/K$ and $\Jbar=(J+K)/K$.  Then
\[
 \ar_{\Jbar}(I/K)\le \ar_J(I).
\]
\end{Lemma}

\begin{proof}
This is \cite[Lemma 2.6]{QNVT}.
\end{proof}

\begin{Lemma}\label{artinreesin}
If $K\subseteq I$, then
$
 \inn(I/K)=\inn(I)/\inn(K)
$
in the associated graded ring of $R/K$.  Moreover, if $\inn(I)=\inn(I')$, then
$
 \ar_J(I)=\ar_J(I').
$
\end{Lemma}

\begin{proof}
The first statement is \cite[Lemma 2.2]{QNVT}; the second follows from \eqref{AR-initial}.
\end{proof}

\begin{Lemma}\label{motphantu}
Let $f\in R$ be a $J$-filter regular element.  Set
\[
 c=\max\{a_J(0:f),\ar_J(f)+1\}.
\]
If $f'=f+\varepsilon$ with $\varepsilon\in J^c$, then 
$f'$ is a $J$-filter regular element with  $0:f'=0:f$
and $\inn(f')=\inn(f)$.
\end{Lemma}

\begin{proof}
This is \cite[Proposition 3.2]{QNVT}.
\end{proof}

\section{Main result}
In this section, we will give the bound of the perturbation index. This result is an improvement of the result of Quy-N. V. Trung \cite[Theorem 3.5]{QNVT}.

First, we will need some useful lemmas.
\begin{Lemma}\label{Loewy}
Let
$$0 \To M_1 \To M_2 \To M_3 \To 0$$
be a short exact sequence of finitely generated $R$-modules and $J$ be an arbitrary ideal of $R$. Suppose that $a_J(M_1)$, $a_J(M_2)$ and $a_J(M_3)$ are finite.
Then
$$\max\{a_J(M_1),a_J(M_3)\} \leq a_J(M_2)\leq a_J(M_1) +a_J(M_3).$$
\end{Lemma}
\begin{proof}
  It follows from the definition of $J$-Loewy length.
\end{proof}
\begin{Lemma}\label{cong1}
  Let $(R,\mm)$ be a Noetherian local ring and $J$ an arbitrary ideal of $R$. Let $f_1,f_2$ be a $J$-filter regular sequence on $R$. Let
  $$a_1=a_J(0:f_1) \ \ \text{and} \ \ a_2=a_J\left(\frac{(f_1):f_2}{(f_1)}\right).$$
  Set
  $$N_0=\max\left\{a_1+a_2,\ar_J(f_1), \ar_J(f_1,f_2)\right\}+1.$$
  For every $\varepsilon \in J^{N_0}$, put $f_1'=f_1+\varepsilon$. Then

    {\rm (i)} $0:f_1'=0:f_1$, $f_1'(0:f_1f_2)=f_1(0:f_1f_2)$, $(f_2):f_1'=(f_2):f_1.$
    
    {\rm (ii)} The following two sequences are short exact
    \begin{equation}\label{exact1}
      0 \To \dfrac{0:f_2}{f_1(0:f_1f_2)} \xrightarrow{i} \dfrac{(f_1):f_2}{(f_1)} \xrightarrow{p} \dfrac{(f_2):f_1}{(f_2)+(0:f_1)} \To 0
    \end{equation}
    and 
    \begin{equation}\label{exact2}
      0 \To \dfrac{0:f_2}{f_1(0:f_1f_2)} \xrightarrow{i'}  \dfrac{(f'_1):f_2}{(f'_1)} \xrightarrow{p'} \dfrac{(f_2):f_1}{(f_2)+(0:f_1)} \To 0.
    \end{equation}
\end{Lemma}
\begin{proof}
(i) is in Lemma \ref{motphantu} and in the proof \cite[Proposition 3.4]{QNVT}.
We will prove (ii).\\
Consider the sequence (\ref{exact1}). We define
$i: \dfrac{0:f_2}{f_1(0:f_1f_2)} \To \dfrac{(f_1):f_2}{(f_1)}$
by $i(x+f_1(0:f_1f_2))=x+(f_1)$ for every $x\in 0:f_2$, and
$
p: \dfrac{(f_1):f_2}{(f_1)} \To \dfrac{(f_2):f_1}{(f_2)+(0:f_1)}
$
by $p(x+(f_1))=y+(f_2)+(0:f_1)$, where $f_2x=f_1y$ for some $y \in (f_2):f_1$. We can check that $i$ is injective, $p$ is surjective, and $\Im(i)=\Ker(p)$. Therefore, (\ref{exact1}) is a short exact sequence.

Consider the sequence (\ref{exact2}). We define
 $i': \dfrac{0:f_2}{f_1(0:f_1f_2)} \To \dfrac{(f'_1):f_2}{(f'_1)} $
by $i'(x+f_1(0:f_1f_2)) = x + (f'_1)$ for every $x \in 0:f_2$. We first check that $i'$ is well-defined. If $x\in0:f_2$, then $f_2x=0\in(f'_1)$, so $x\in(f'_1):f_2$. Moreover, if
$$
x+f_1(0:f_1f_2)=x'+f_1(0:f_1f_2),
$$
then $x-x'\in f_1(0:f_1f_2)=f'_1(0:f_1f_2)\subseteq(f'_1)$. Hence $x+(f'_1)=x'+(f'_1)$, and therefore $i'$ is well-defined. We have that $i'$ is injective, since
 \begin{align*}
   \Ker(i') & = \left\{ x + f_1(0:f_1f_2) \mid x \in 0:f_2 \ \text{and} \ x \in (f_1') \right\} \\
    & = \left\{ x + f_1(0:f_1f_2) \mid x \in f_1'(0:f_1'f_2) =  f_1'(0:f_1f_2) \right\} \quad (\text{since} \quad 0:f_1'=0:f_1) \\
    & =  \left\{ x + f_1(0:f_1f_2) \mid x \in f_1(0:f_1f_2) \right\} \quad (\text{since} \quad f_1'(0:f_1f_2) =  f_1(0:f_1f_2)) \\
    & = \overline{0}.
 \end{align*}
We also define
 $ p': \dfrac{(f'_1):f_2}{(f'_1)} \To \dfrac{(f_2):f_1}{(f_2)+(0:f_1)} $
by $p'(x+(f'_1)) = y + (f_2)+(0:f_1)$ for every $x\in (f'_1):f_2$, where $f_2x=f'_1y$ for some $y \in (f_2):f_1$. We can check that $p'$ is well-defined. We have that $p'$ is surjective. Indeed, for every $y + (f_2)+(0:f_1)$ with $y \in (f_2):f_1=(f_2):f'_1$, there exists $x \in R$ such that $f_2x = f_1'y$. It follows that $x \in (f_1'):f_2$ and $p'(x+(f'_1)) = y + (f_2)+(0:f_1)$.

We now show that $\Im(i')=\Ker(p')$. Clearly, $\Im(i')\subseteq\Ker(p')$. Conversely, let $x+(f'_1)\in\Ker(p')$. Then $f_2x=f'_1y$ for some $y\in (f_2)+(0:f_1)$. Write $y=f_2z+w$, where $w\in0:f_1=0:f'_1$. Hence
$$f_2x=f'_1(f_2z+w)=f_2f'_1z,$$
and thus $x-f'_1z\in0:f_2$. Since $x+(f'_1)=(x-f'_1z)+(f'_1)$, we obtain $x+(f'_1)\in\Im(i')$. Therefore, $\Im(i')=\Ker(p')$. Hence (\ref{exact2}) is a short exact sequence. The proof is complete. 
\end{proof}
From Lemma \ref{cong1} (ii), we set $$A :=  \dfrac{0:f_2}{f_1(0:f_1f_2)}\ \ \text{and} \ \ C:=\dfrac{(f_2):f_1}{(f_2)+(0:f_1)}.$$ Then 
$$\xi: 0 \To A \To \dfrac{(f_1):f_2}{(f_1)} \To C \To 0$$
and 
$$\xi': 0 \To A \To \dfrac{(f'_1):f_2}{(f'_1)} \To C \To 0$$
are two \emph{extensions} of $A$ by $C$ \cite[the definition after Example 7.23]{Rotman}. Two extensions $\xi$ and $\xi'$ are \emph{equivalent} if there exists a map $\varphi:  \dfrac{(f_1):f_2}{(f_1)} \To \dfrac{(f'_1):f_2}{(f'_1)}$ making
the following diagram commute: 

\[
\begin{tikzcd}[column sep=large]
0 \arrow[r] 
& A \arrow[r,"i"] \arrow[d,"\operatorname{id}_A"']
& \dfrac{(f_1):f_2}{(f_1)} \arrow[r,"p"] \arrow[d,"\varphi"]
& C \arrow[r] \arrow[d,"\operatorname{id}_C"] 
& 0\\
0 \arrow[r] 
& A \arrow[r,"i'"']
& \dfrac{(f'_1):f_2}{(f'_1)} \arrow[r,"p'"']
& C \arrow[r] 
& 0
\end{tikzcd}
\]
We denote the equivalence class of an extension $\xi$ by $[\xi]$ and we define 
$$e(C,A) = \left\{[\xi] \mid \xi \ \text{is an extension of $A$ by $C$}\right\}.$$
If $\xi$ and $\xi'$ are equivalent then $\varphi$ is an isomorphism by the Five Lemma. 
Given a projective resolution $\mathbf{P}$ of $C$, form the diagram

\[
\begin{tikzcd}[column sep=large]
 \arrow[r] 
& P_2 \arrow[r,"d_2"] \arrow[d,dotted]
& P_1 \arrow[r,"d_1"] \arrow[d,,dotted,"\alpha_1"]
& P_0 \arrow[r,"\pi"] \arrow[d,dotted,"\alpha_0"] 
& C\arrow[r] \arrow[d, "\operatorname{id}_C"]
& 0\\
\arrow[r] 
& 0 \arrow[r]
& A \arrow[r,"i"']
& \dfrac{(f_1):f_2}{(f_1)} \arrow[r,"p"'] 
& C\arrow[r]
&0
\end{tikzcd}
\tag{A}
\]
By the Comparison Theorem \cite[Theorem 6.16]{Rotman}, there exist dotted arrows which comprise a chain map $(\alpha_n): \mathbf{P} \To \xi$ over $\operatorname{id}_C$. In particular, the first component $\alpha_1: P_1 \To A$ satisfies $\alpha_1d_2=0$, thus $d_2^*(\alpha_1)=0$, it follows that $\alpha_1 \in \Ker(d_2^*)$ and $\alpha_1$ is a cocycle.
By \cite[Theorem 7.30]{Rotman}, there is an isomorphism 
$$\psi : e(C,A) \To \Ext_R^1(C,A)$$
with $\psi([\xi]) = \overline{\alpha_1}$, where $ \overline{\alpha}_1 = \alpha_1 + \Im(d_1^*) \in \Ker(d_2^*)/\Im(d_1^*)=\Ext_R^1(C,A)$.

In the following proposition, we construct two explicit cocycles
$\alpha_1$ and $\alpha_1'$ representing $\xi$ and $\xi'$, respectively,
after perturbing $f_1$ to $f_1'$. We then show that these two cocycles
are equal. From this, we obtain
$$
\frac{(f_1):f_2}{(f_1)}
\cong
\frac{(f_1'):f_2}{(f_1')}.
$$
\begin{Proposition}
  \label{perturb1}
  Let $(R,\mm)$ be a Noetherian local ring and $J$ an arbitrary ideal of $R$. Let $f_1,f_2$ be a $J$-filter regular sequence on $R$. Let
  $$a_1=a_J(0:f_1) \ \ \text{and} \ \ a_2=a_J\left(\dfrac{(f_1):f_2}{(f_1)}\right).$$
  Set
  $$N=\max\left\{a_1+2a_2,\ar_J(f_1), \ar_J(f_1,f_2)\right\}+1.$$
  Then for every $\varepsilon \in J^N$, $f'_1=f_1+\varepsilon$, we have

    {\rm (i)} $\dfrac{(f_1'):f_2}{(f_1')} \cong \dfrac{(f_1):f_2}{(f_1)}.$

    {\rm (ii)} $a_J\left(\dfrac{(f_1'):f_2}{(f_1')}\right)=a_J\left(\dfrac{(f_1):f_2}{(f_1)}\right)=a_2.$

    {\rm (iii)} $f_1',f_2$ is a $J$-filter regular sequence.

    {\rm (iv)} $\inn(f_1',f_2)=\inn(f_1,f_2)$.
\end{Proposition}
\begin{proof}
We only prove (i) and (iv), since parts (ii), (iii) follow from (i).

First, we construct the cocycle $\alpha_1$ appearing in diagram (A).
Let
\[
\mathbf{P}:\qquad
\cdots \longrightarrow P_2
\xrightarrow{d_2}
P_1
\xrightarrow{d_1}
P_0
\xrightarrow{\pi}
C
\longrightarrow 0
\]
be a projective resolution of $C$.
Let
$$
q:(f_2):f_1 \twoheadrightarrow
C=\frac{(f_2):f_1}{(f_2)+(0:f_1)}
$$
be the canonical quotient map.
Since $P_0$ is a projective $R$-module and $q$ is surjective,
there exists an $R$-linear lifting
\[
\rho:P_0\longrightarrow (f_2):f_1
\]
such that
$
q\rho=\pi.
$
Therefore, the following diagram is commutative.
\[
\begin{tikzcd}[column sep=large,row sep=large]
& P_0 \arrow[dl,dashed,"\rho"'] \arrow[d,"\pi"] \\
(f_2):f_1 \arrow[r,two heads,"q"'] & C.
\end{tikzcd}
\]
Since
$
f_1\bigl((f_2):f_1\bigr)\subseteq (f_2),
$
multiplication by $f_1$ yields an $R$-linear map
\[
f_1\rho:P_0\longrightarrow (f_2).
\] 
Let
$$
\mu_{f_2}:R\twoheadrightarrow (f_2)
$$
with $s\longmapsto f_2s$
be the multiplication map. Since \(f_1\rho(P_0)\subseteq (f_2)\) and
\(P_0\) is a projective \(R\)-module, the map
$
f_1\rho:P_0\longrightarrow (f_2)
$
lifts through \(\mu_{f_2}\). Hence there exists an \(R\)-linear map
\[
\sigma:P_0\longrightarrow R
\]
such that
$
\mu_{f_2}\sigma=f_1\rho
$
or the following diagram is commutative
\[
\begin{tikzcd}[column sep=large,row sep=large]
& R \arrow[d,two heads,"\mu_{f_2}"] \\
P_0 \arrow[ur,dashed,"\sigma"] \arrow[r,"f_1\rho"'] & (f_2).
\end{tikzcd}
\]
or equivalently,
$
f_2\sigma=f_1\rho.
$
Since
\[
q\rho d_1=\pi d_1=0,
\]
we have
$
\rho d_1(P_1)\subseteq \ker q
=(f_2)+(0:f_1).
$
Consider the surjective $R$-linear map
\[
\phi:R\oplus (0:f_1)\twoheadrightarrow (f_2)+(0:f_1),
\qquad
\phi(a,v_0)=f_2a+v_0.
\]
Since $P_1$ is a projective $R$-module and
$
\rho d_1(P_1)\subseteq (f_2)+(0:f_1),
$
the map
$$
\rho d_1:P_1\longrightarrow (f_2)+(0:f_1)
$$
lifts through $\phi$. Hence there exist $R$-linear maps
\[
u:P_1\longrightarrow R,
\qquad
v:P_1\longrightarrow (0:f_1)
\]
such that
$
\phi\circ (u,v)=\rho d_1
$
or the following diagram is commutative
\[
\begin{tikzcd}[column sep=large,row sep=large]
& R\oplus (0:f_1) \arrow[d,two heads,"\phi"] \\
P_1 \arrow[ur,dashed,"{(u,v)}"] \arrow[r,"\rho d_1"'] & (f_2)+(0:f_1).
\end{tikzcd}
\]
Equivalently,
$$
\rho d_1=f_2u+v.
$$
Set
\[
z:=\sigma d_1-f_1u:P_1\longrightarrow R.
\]
Then
\begin{align*}
f_2z
&=f_2\sigma d_1-f_2f_1u\\
&=f_1\rho d_1-f_1f_2u
&&(\text{since } f_2\sigma=f_1\rho)\\
&=f_1(f_2u+v)-f_1f_2u
&&(\text{since } \rho d_1=f_2u+v)\\
&=f_1v\\
&=0
&& (\text{since } v(P_1)\subseteq (0:f_1)).
\end{align*}
Therefore
$
z(P_1)\subseteq 0:f_2.
$
Let
\[
q_A:0:f_2\twoheadrightarrow A,\qquad
q_A(a)=a+f_1(0:f_1f_2),
\]
denote the canonical quotient map and set
 $$\alpha_1:= q_A \circ z:  P_1 \To A. $$
We claim that $\alpha_1$ is a cocycle of extension $\xi$. Indeed, from $f_2\sigma = f_1 \rho$, we get $\sigma(P_0)\subseteq (f_1):f_2$, define 
$$\alpha_0: P_0 \To \dfrac{(f_1):f_2}{(f_1)},$$ with $\alpha_0(x) = \sigma(x) + (f_1)$. By the definition of $p$, we get
$$p\alpha_0= q\rho = \pi.$$
On the other hand, $\sigma d_1 = z + f_1u$, for every $x \in P_1$  we have
$$\alpha_0d_1(x) = \sigma d_1(x) + (f_1) = z(x) + (f_1) = i \alpha_1(x).$$
Thus, $\alpha_0d_1 = i \alpha_1$. Hence, 
$i\alpha_1d_2 = \alpha_0d_1d_2 =0$ and $i$ is injective, so $\alpha_1d_2 =0$.
We conclude that  $\alpha_1$ is a cocycle of the extension $\xi$. 

Second, we construct the cocycle $\alpha'_1$ appearing in diagram (B).
\[
\begin{tikzcd}[column sep=large]
 \arrow[r] 
& P_2 \arrow[r,"d_2"] \arrow[d,dotted]
& P_1 \arrow[r,"d_1"] \arrow[d,,dotted,"\alpha'_1"]
& P_0 \arrow[r,"\pi"] \arrow[d,dotted,"\alpha'_0"] 
& C\arrow[r] \arrow[d, "\operatorname{id}_C"]
& 0\\
\arrow[r] 
& 0 \arrow[r]
& A \arrow[r,"i'"']
& \dfrac{(f'_1):f_2}{(f'_1)} \arrow[r,"p'"'] 
& C\arrow[r]
&0
\end{tikzcd}
\tag{B}
\]
From the short exact sequence (\ref{exact1}), since
\[
J^{a_2}\left(\dfrac{(f_1):f_2}{(f_1)}\right)=0,
\]
by Lemma \ref{Loewy}, it follows that $J^{a_2}C=0$, that is,
$
J^{a_2}((f_2):f_1)\subseteq (f_2)+(0:f_1).
$
Hence
\[
J^{a_1+a_2}((f_2):f_1)
\subseteq
J^{a_1}((f_2)+(0:f_1))
\subseteq (f_2).
\]
Since $\varepsilon\in J^N\subseteq J^{a_1+2a_2}$, we can write
$$
\varepsilon=\sum_{\nu=1}^t a_\nu\eta_\nu,
$$
for some $t$ and $a_\nu\in J^{a_2}$, $\eta_\nu\in J^{a_1+a_2}$.
For each $\nu=1,\ldots,t$, we have
$
\eta_\nu((f_2):f_1)\subseteq (f_2).
$
Since $P_0$ is a projective $R$-module, for each $\nu$ there exists a map
$
\tau_\nu:P_0\to R
$
such that
$
f_2\tau_\nu=\eta_\nu\rho.
$
Therefore, the following diagram is commutative
\[
\begin{tikzcd}[column sep=large,row sep=large]
& R \arrow[d,two heads,"\mu_{f_2}"] \\
P_0 \arrow[ur,dashed,"\tau_\nu"] \arrow[r,"\eta_\nu\rho"'] & (f_2)
\end{tikzcd}
\]
Let
\[
\tau:=\sum_{\nu=1}^t a_\nu\tau_\nu.
\]
Then
\[
f_2\tau
=
f_2\left(\sum_{\nu=1}^t a_\nu\tau_\nu\right)
=
\sum_{\nu=1}^t a_\nu f_2\tau_\nu
=
\sum_{\nu=1}^t a_\nu\eta_\nu\rho
=
\varepsilon\rho.
\]
Set
$
\sigma':=\sigma+\tau.
$
Then
$
f_2\sigma'
=
f_2\sigma+f_2\tau
=
f_1\rho+\varepsilon\rho
=
f_1'\rho.
$
It follows that
\[
\sigma'(P_0)\subseteq (f_1'):f_2.
\]
Let
\[
u:P_1\to R
\qquad\text{and}\qquad
v:P_1\to (0:f_1)=(0:f_1')
\]
be as above and set
$
z':=\sigma'd_1-f_1'u:P_1\to R.
$
We have
\begin{align*}
f_2z'
&=f_2\sigma'd_1-f_2f_1'u\\
&=f_1'\rho d_1-f_1'f_2u 
&&(\text{since } f_2\sigma'=f_1'\rho)\\
&=f_1'(f_2u+v)-f_1'f_2u
&&(\text{since } \rho d_1=f_2u+v)\\
&=f_1'v\\
&=0 && (\text{since } v(P_1)\subseteq (0:f'_1)).
\end{align*}
Thus
$
z'(P_1)\subseteq 0:f_2.
$
Let
$
q_A:0:f_2\twoheadrightarrow A
$
denote the canonical quotient map and set
$$
\alpha_1':=q_A\circ z':P_1\longrightarrow A.
$$
Define
$
\alpha_0':P_0\longrightarrow  \dfrac{(f'_1):f_2}{(f'_1)}
$
by
$
\alpha_0'(x):=\sigma'(x)+(f_1').
$
Since
$
f_2\sigma'=f_1'\rho,
$
we have
$
p'\alpha_0'=q\rho=\pi.
$

As in the construction of $\alpha_1$, we have
$
i'\alpha_1'=\alpha_0'd_1.
$
Hence
$
i'\alpha_1'd_2=\alpha_0'd_1d_2=0.
$
Since $i'$ is injective,
$
\alpha_1'd_2=0.
$
Therefore $\alpha_1'$ is a cocycle of the extension class
$\xi'$.

Finally, we will prove $\alpha_1 = \alpha_1'$. For each
$\nu=1,\ldots,t$, set
\[
w_\nu:=\tau_\nu d_1-\eta_\nu u:P_1\to R.
\]
Since $
\rho d_1=f_2u+v,
$
we have
\begin{align*}
f_2w_\nu
&=f_2\tau_\nu d_1-\eta_\nu f_2u\\
&=\eta_\nu\rho d_1-\eta_\nu f_2u\\
&=\eta_\nu(\rho d_1-f_2u)\\
&=\eta_\nu v.
\end{align*}
Since
$
v(P_1)\subseteq 0:f_1
$
and
$
\eta_\nu\in J^{a_1+a_2}\subseteq J^{a_1},
$
while
$
J^{a_1}(0:f_1)=0,
$
we obtain
$
\eta_\nu v=0.
$
Hence
$
w_\nu(P_1)\subseteq 0:f_2,
$
so that
$
q_A\circ w_\nu:P_1\to A
$
is well-defined.\\
Moreover,
\begin{align*}
z'-z
&=(\sigma'd_1-f_1'u)-(\sigma d_1-f_1u)\\
&=\tau d_1-\varepsilon u\\
&=\sum_{\nu=1}^t
a_\nu(\tau_\nu d_1-\eta_\nu u)\\
&=\sum_{\nu=1}^t a_\nu w_\nu.
\end{align*}
Therefore,
\[
\alpha_1'-\alpha_1
=
q_A\circ(z'-z)
=
\sum_{\nu=1}^t a_\nu q_A\circ w_\nu.
\]
From the short exact sequence (\ref{exact1}), since
$
J^{a_2}\left(\dfrac{(f_1):f_2}{(f_1)}\right)=0,
$
by Lemma \ref{Loewy},
we have
$
J^{a_2}A=0.
$
Since $a_\nu\in J^{a_2}$ for every $\nu$, it follows that
\[
a_\nu q_A\circ w_\nu=0.
\]
Consequently,
\[
\alpha_1'=\alpha_1.
\]
Thus $\xi$ and $\xi'$ are represented by the same cocycle in
$\Hom_R(\mathbf{P},A)$. Hence
$
[\xi]=[\xi']$ in
$
\Ext_R^1(C,A).
$
Therefore the two extensions are equivalent. In particular, there exists
an $R$-module isomorphism $\varphi: \dfrac{(f_1):f_2}{(f_1)}\to \dfrac{(f'_1):f_2}{(f'_1)}$.

(iv). By the proof of \cite[Proposition 3.4]{QNVT}, we have $a_J\left(\dfrac{(f_2):f_1}{(f_2)}\right) \leq a_1+a_2$. Thus $f_1$ is a $J$-filter regular element in $R'=R/(f_2)$. On the other hand, $\ar_{JR'}(f_1R') \leq \ar_J(f_1,f_2)$ by Lemma \ref{artinreesthuong}. Therefore, 
$$N \geq \max\{a_J(0_{R'}:f_1R'),\ar_{JR'}(f_1R')+1\}.$$
Applying Lemma \ref{motphantu}, we get $$\inn(f_1'R') = \inn(f_1R').$$
By Lemma \ref{artinreesin}, this implies
$$\inn(f_1',f_2) = \inn(f_1,f_2).$$ 
The proof is completed.
\end{proof}
In the following lemma, we first perturb the element $f_1$. 
\begin{Lemma}\label{extend}
  Let $(R,\mm)$ be a Noetherian local ring and $J$ an arbitrary ideal of $R$. Let $f_1,\ldots,f_r$ be a $J$-filter regular sequence on $R$. Let
  $$a_i=a_J\left(\dfrac{(f_1,\ldots,f_{i-1}):f_i}{(f_1,\ldots,f_{i-1})}\right)$$
  for $i=1,\ldots,r$. Set
$
T= \max_{2\le i\le r} \left\{ a_i-\sum_{j=i+1}^r a_j \right\} \quad \text{and} \quad T=0 \text{ if } r=1,
$
  $$N=\max\{a_1+\cdots+a_r+ T, \ar_J(f_1),\ldots,\ar_J(f_1,\ldots,f_r)\}+1.$$
  Then for every $\varepsilon \in J^N$ and $f_1'=f_1+\varepsilon$, we have

  {\rm (i)} $f_1',f_2,\ldots,f_{r}$ is a $J$-filter regular sequence on $R$ with
   $$a_J\left(\dfrac{(f_1',f_2,\ldots,f_{i-1}):f_{i}}{(f_1',f_2,\ldots,f_{i-1})}\right)=a_{i}$$
   for all $i=1,\ldots,r$.

  {\rm (ii)} $\inn(f_1',f_2,\ldots,f_i)=\inn(f_1,f_2,\ldots,f_i)$ for all $i=1,\ldots,r$.

\end{Lemma}
\begin{proof}
(i). First, we will prove that $f_1$ is a $J$-filter regular element on $\overline{R}_i=R/(f_2,\ldots,f_i)$ for $i=1,\ldots,r$.  The case $r=1$ is proved in Lemma \ref{motphantu}. For $r \geq 2$, by the proof of \cite[Proposition 3.4 and Theorem 3.5]{QNVT} we have
  $$a_J\left(\frac{(f_2,\ldots,f_i):f_1}{(f_2,\ldots,f_i)}\right)\leq a_1+\cdots + a_i.$$
  From this it follows that $f_1$ is a $J$-filter regular element on $\overline{R}_i=R/(f_2,\ldots,f_i)$ for $i=1,\ldots,r$. Notice that $(f_2,\ldots,f_i)=0$ if $i=1$.

  For $i = 2,\ldots,r$, by hypothesis $f_{i}$ is a $J$-filter regular element on $R/(f_1,\ldots,f_{i-1})$. Therefore, $f_1, f_{i}$ is a $J$-filter regular sequence on $\overline{R}_{i-1}$ with
  $$a_J(0_{\overline{R}_{i-1}}:f_1)=a_J\left(\frac{(f_2,\ldots,f_{i-1}):f_1}{(f_2,\ldots,f_{i-1})}\right)\leq a_1+\cdots + a_{i-1}$$
  and
  $$a_J\left(\frac{f_1\overline{R}_{i-1}:f_{i}}{f_1\overline{R}_{i-1}}\right)=a_J\left(\frac{(f_1,\ldots,f_{i-1}):f_{i}}{(f_1,\ldots,f_{i-1})}\right)=a_{i}.$$
  On the other hand, by Lemma \ref{artinreesthuong} $$\ar_{J\overline{R}_{i-1}}(f_1\overline{R}_{i-1}) \leq \ar_J(f_1,\ldots,f_{i-1}) \ \ \text{and} \ \ \ar_{J\overline{R}_{i-1}}((f_1,f_{i})\overline{R}_{i-1}) \leq \ar_J(f_1,\ldots,f_{i}).$$
  It follows that
  $$N \geq \max\{a_J(0_{\overline{R}_{i-1}}:f_1)+2a_J(f_1\overline{R}_{i-1}:f_{i}/f_1\overline{R}_{i-1}),
  \ar_{J\overline{R}_{i-1}}(f_1\overline{R}_{i-1}), \ar_{J\overline{R}_{i-1}}((f_1,f_{i})\overline{R}_{i-1}) \}+1.$$
  By Proposition \ref{perturb1} (ii) and (iii), we have $f_1',f_{i}$ is a $J$-filter regular sequence on $\overline{R}_{i-1}$ with
  $$a_J\left(\frac{f_1'\overline{R}_{i-1}:f_{i}}{f_1'\overline{R}_{i-1}}\right)=
  a_J\left(\frac{f_1\overline{R}_{i-1}:f_{i}}{f_1\overline{R}_{i-1}}\right)=a_J\left(\frac{(f_1,\ldots,f_{i-1}):f_{i}}{(f_1,\ldots,f_{i-1})}\right)=a_{i}.$$
  It follows that $f_1',f_2,\ldots,f_r$ is a $J$-filter regular sequence on $R$ with 
   $$a_J\left(\frac{(f_1',f_2,\ldots,f_{i-1}):f_{i}}{(f_1',\ldots,f_{i-1})}\right)=a_{i}$$
  for $i=2,\ldots,r$. In the case $i=1$, it follows from Lemma \ref{motphantu}
  $$a_J(0:f_1') = a_J(0:f_1) =a_1.$$

  (ii). By Proposition \ref{perturb1} (iv), for $i=2,\ldots, r$, we have
   $$\inn((f_1',f_{i})\overline{R}_{i-1})=\inn((f_1,f_{i})\overline{R}_{i-1}).$$
  Hence, by Lemma \ref{artinreesin}
  $$\inn(f_1',f_2,\ldots,f_{i})=\inn(f_1,\ldots,f_{i})$$
  for $i=2,\ldots,r$.
  By Lemma \ref{motphantu}, we have $f_1'$ is a $J$-filter regular element in $R$ and $\inn(f_1')=\inn(f_1)$. Therefore,
  $$\inn(f_1',f_2,\ldots,f_i) = \inn(f_1,f_2,\ldots,f_i)$$
  for $i=1,\ldots, r$. The proof is now complete.
\end{proof}
After perturbing the first element $f_1$, we use induction on $r$ to prove the main result of this paper in the following theorem. 
\begin{Theorem}\label{main}
  Let $(R,\mm)$ be a Noetherian local ring and $J$ an arbitrary ideal of $R$. Let $f_1,\ldots,f_r$ be a $J$-filter regular sequence on $R$ and $I=(f_1,\ldots,f_r)$. Let
  $$a_i=a_J\left(\frac{(f_1,\ldots,f_{i-1}):f_i}{(f_1,\ldots,f_{i-1})}\right)$$
  for $i=1,\ldots,r$.
Set
$
T= \max_{2\le i\le r} \left\{ a_i-\sum_{j=i+1}^r a_j \right\} \quad \text{and} \quad T=0 \text{ if } r=1,
$
  $$N=\max\{a_1+\cdots+a_r+ T, \ar_J(f_1),\ldots,\ar_J(f_1,\ldots,f_r)\}+1.$$
  Then for every $I'=(f'_1,\ldots,f_r')$ where $f_i'=f_i+\varepsilon_i$, $\varepsilon_i \in J^N$ for $i=1,\ldots,r$, we have

    {\rm (i)} $f_1',\ldots,f_r'$ is a $J$-filter regular sequence with
    $$a_J\left(\frac{(f_1',\ldots,f_{i-1}'):f_i'}{(f_1',\ldots,f_{i-1}')}\right)=a_i$$
    for $i=1,\ldots,r$.

    {\rm (ii)} $\inn(f_1',\ldots,f_i')=\inn(f_1,\ldots,f_i)$ for $i=1,\ldots,r$.

    {\rm (iii)} In particular, $\inn(I') = \inn(I)$, $\gr_J(R/I') \cong \gr_J(R/I)$,  $\ar_J(I')=\ar_J(I)$.

\end{Theorem}
\begin{proof}
We notice that 
$$a_1+\cdots+a_r+ T = \max_{2 \leq i \leq r}\{a_1+a_2+\cdots+a_{i-1}+2a_i\}.$$
Therefore,
$$N=\max\{a_1+2a_2,a_1+a_2+2a_3,\ldots,a_1+\cdots+a_{r-1}+2a_r,\ar_J(f_1),\ldots,\ar_J(f_1,\ldots,f_r)\}+1.$$
  (i). We will prove by induction on $r$. The case $r=1$ is Lemma \ref{motphantu}.
  For $r \geq 2$, 
  by Lemma \ref{extend} the sequence  $f_1',f_2,\ldots,f_r$ is a $J$-filter regular sequence on $R$ with 
  $$a_J\left(\dfrac{(f_1',f_2,\ldots,f_{i-1}):f_{i}}{(f_1',f_2,\ldots,f_{i-1})}\right)=a_{i}$$
   for $i=1,\ldots,r$
  and
  \begin{equation}\label{in1}
    \inn(f_1',f_2,\ldots,f_{i}) = \inn(f_1,f_2,\ldots,f_{i})
  \end{equation}
  for $i=1,\ldots,r$. It follows that
  $$\ar_J(f_1',f_2,\ldots,f_{i}) = \ar_J(f_1,f_2,\ldots,f_{i})$$
  for $i=1,\ldots,r$.
   Set $R'=R/(f_1')$, then $f_2,\ldots,f_r$ is a $J$-filter regular sequence on $R'$. 
   Let 
   $$a_2':=a_J(0_{R'}:f_2) \quad \text{and} \quad a_i':=a_J\left(\frac{(f_2,\ldots,f_{i-1})R':f_{i}}{(f_2,\ldots,f_{i-1})R'}\right) \quad \text{for} \quad 3\leq i \leq r$$
   and 
   $$T'= \max_{3\le i\le r} \left\{ a'_i-\sum_{j=i+1}^r a'_j \right\}.$$
   By the argument above, we have $a_i'=a_i$ for $i=2,\ldots,r$ and $T' \leq T$. 
   By Lemma 2.1, $$\operatorname{ar}_{JR'}((f_2,\dots,f_i)R')\le\operatorname{ar}_J(f_1',f_2
,\dots,f_i)=\operatorname{ar}_J(f_1,\dots,f_i)$$
for $i=2,\dots,r$.
   
   Therefore, 
   $$N \geq \max\{a'_2+\cdots+a'_r+T', \ar_{JR'}(f_2R'),\ldots, \ar_{JR'}(f_2,\ldots,f_r)R'\}+1.$$
Using induction on $r$, we may assume that
\begin{enumerate}[label=\textup{(\roman*')}]
\item $f_2',\ldots,f_r'$ is a $J$-filter regular sequence in $R'$ with
$$a_J\left(\dfrac{(f_2',\ldots,f'_{i-1})R':f_{i}'}{(f_2',\ldots,f'_{i-1})R'}\right) = a_i$$
for $i=2,\ldots, r$.
\item $\inn((f_2',\ldots,f'_{i})R') = \inn((f_2,\ldots,f_{i})R')$ for $i=2,\ldots, r.$ 
\end{enumerate}   
  
  By Lemma \ref{motphantu}, we have $f'_1$ is a $J$-filter regular element with $(0:f_1')=(0:f_1)$. From this and (i'), we get
  $f_1',\ldots,f_r'$ is a $J$-filter regular sequence on $R$ with
  $$a_J\left(\dfrac{(f_1',f_2',\ldots,f'_{i-1}):f_{i}'}{(f_1',f_2',\ldots,f'_{i-1})}\right) = a_{i}$$
  for $i=1,\ldots, r$. By Lemma \ref{artinreesin} and (ii') we have
  \begin{equation}\label{inall}
    \inn(f_1',f_2',\ldots,f_{i}') = \inn(f_1',f_2,\ldots,f_{i})
  \end{equation}
  for $i=1,\ldots,r$.
From (\ref{in1}) and (\ref{inall}), we have
  $$\inn(f_1',\ldots,f_{i}') = \inn(f_1,\ldots,f_{i})$$ 
    for $i=1,\ldots,r$. 
  
  (iii). By the isomorphism (\ref{isomorphism}), we have $\gr_J(R/I) \cong \gr_J(R)/\inn(I)$ and $\gr_J(R/I') \cong \gr_J(R)/\inn(I')$. From (ii), $\inn(I')=\inn(I)$. It follows that $$\gr_J(R/I') \cong \gr_J(R/I).$$
By (\ref{AR-initial}), $\ar_J(I)=d(\inn(I))$ and $\ar_J(I')=d(\inn(I'))$. From (ii), we get $\ar_J(I')=\ar_J(I)$.
\end{proof}

\begin{Remark}\label{Remark}
Let $r\ge2$. The bounds $N$ and $N_{\mathrm{Tan}}$ are not comparable in general. Put
$\rho=\max_{1\le i\le r}\operatorname{ar}_J(f_1,\dots,f_i)$. Then $N\le N_{\mathrm{Tan}}$ unless
\[a_2>2a_3+4a_4+\dots+2^{r-2}a_r\quad\text{and}\quad a_1+2a_2>\rho,\]
and in that case $N>N_{\mathrm{Tan}}$. For instance, if $r=2$, $a_1=a_2=1$, $\rho=0$, then
$N=4>3=N_{\mathrm{Tan}}$, whereas if $r\ge4$, $a_1=\dots=a_r=1$, $\rho=0$, then
$N=r+2<2^{r-1}+1=N_{\mathrm{Tan}}$ (Example \ref{example}).
\end{Remark}
We use \cite[Example 3.5]{Tanarxiv} to illustrate the improvement of our bound over the known results.  
\begin{Example}\label{example}

Let $r \geq 2$ and $k$ be a field and 
$$
 A=k[[x_1,\ldots,x_r,y]],
 \qquad
 \mathfrak q=(x_1,\ldots,x_r,y).
$$
Regard $k=A/\mathfrak q$ as an $A$-module and form the idealization
\[
 R=A\ltimes k=A\oplus k,
\]
with multiplication
\[
 (a,\lambda)(b,\mu)=(ab,\overline a\mu+\overline b\lambda),
\]
where bars denote residue classes in $k=A/\mathfrak q$. 
Then $R$ is a Noetherian local ring with maximal ideal $\mathfrak q\ltimes k$.

Set
\[
 J=((y,0))R
 \qquad\text{and}\qquad
 f_i=(x_i,0),\quad i=1,\ldots,r.
\]
For $I_i=(x_1,\ldots,x_i)A$ and $F_i=(f_1,\ldots,f_i)R$, a direct computation gives
\[
 F_i=I_i\oplus0,
 \qquad
 a_i=a_J\left(\frac{F_{i-1}:f_i}{F_{i-1}}\right)=1,
 \qquad
 \ar_J(F_i)=0
 \qquad(i=1,\ldots,r).
\]
Consequently, the bound of Quy and N. V. Trung \cite{QNVT} gives
$$
N_{\QNVT}=2^r,
$$
whereas the bound of Tan \cite{Tanarxiv} gives
$$
N_{\mathrm{Tan}}=2^{r-1}+1.
$$
On the other hand, the bound in Theorem \ref{main} gives
$$
N=r+2.
$$
Thus, the dependence on $r$ is reduced from exponential to linear.
\end{Example}
If we set $b := \max_{2 \leq i \leq r}\{a_i\}$ and $b=0$ if $r=1$ then
$$
T= \max_{2\le i\le r} \left\{ a_i-\sum_{j=i+1}^r a_j \right\} \leq b.
$$
Therefore, we get the following corollary.
\begin{Corollary}
Let $(R,\mm)$ be a Noetherian local ring and $J$ an arbitrary ideal of $R$  and
$f_1,\ldots,f_r$ be a $J$-filter regular sequence in $R$.  For $i=1,\ldots,r$, let $$a_i =  a_J\!\left(
 \dfrac{(f_1,\ldots,f_{i-1}):f_i}{(f_1,\ldots,f_{i-1})}
 \right). $$ 
Set
$b=\max_{2\leq i \leq r}\{a_i\}$ and
$$
 N_1
 =\max\left\{
 a_1+\cdots+a_r+b,
 \ar_J(f_1),\ldots,\ar_J(f_1,\ldots,f_r)
 \right\}+1.
$$
For all $\varepsilon_1,\ldots,\varepsilon_r\in J^{N_1}$ and $f_i'=f_i+\varepsilon_i$ for $i=1,\ldots,r$. Then
    \begin{enumerate}[label=\textup{(\roman*)}]
\item $f_1',\ldots,f_r'$ is a $J$-filter regular sequence with  $$a_J\left(\dfrac{(f_1',\ldots,f_{i-1}'):f_i')}{(f_1',\ldots,f_{i-1}')}\right) = a_i$$ for $i=1,\ldots,r$.
\item  For  $i=1,\ldots,r$, we have
$
 \inn(f_1',\ldots,f_i')=\inn(f_1,\ldots,f_i).$

\item In particular, 
$
 \gr_J(R/I')\cong\gr_J(R/I)$ and  $\ar_J(I')=\ar_J(I).
$
\end{enumerate}
    \end{Corollary}
    \begin{proof}
Since $T \leq \max_{2 \leq i \leq r}\{a_i\}$, we get $N\leq N_1$. It follows that $J^{N_1}\subseteq J^N$. Hence $\varepsilon\in J^{N_1}$ implies $\varepsilon\in J^N$ and the corollary follows from Theorem \ref{main}.
Theorem
    \end{proof}

\section{Consequences and applications}\label{application}
In this section, we recover the stability of the Bass and Betti numbers proved in \cite{VDT}. We recall \cite[Proposition 3.4]{VDT} when $M=R$.

        \begin{Proposition}\label{length one}
	Let $(R,\mm)$ be a Noetherian local ring and $J$ an arbitrary ideal of $R$. Let \( f \in R \) be a \( J \)-filter regular element on $R$. Set
	\[
	c = \max\left\{ a_J(0 : f),\ \ar_J(f) + 1 \right\} + 1.
	\]
	For every \( \varepsilon \in J^c \), let \( f' = f + \varepsilon \). Then, for each \( j \geq 0 \), we have:
	\begin{enumerate}\setlength{\itemsep}{0.3em}
		\item[\rm (i)] \( \beta^R_j(R/(f)) = \beta^R_j(R/(f')) \),
		\item[\rm (ii)] \( \mu^j_R(R/(f)) = \mu^j_R(R/(f')) \).
	\end{enumerate}
        \end{Proposition}

Combining Proposition \ref{length one} with the proof of Theorem \ref{main}, we obtain the following theorem.

        \begin{Theorem}\label{arbilength}
Let $(R,\mm)$ be a Noetherian local ring and $J$ an arbitrary ideal of $R$. Let \( f_1,\ldots,f_r \in R \) be a \( J \)-filter regular sequence on $R$. For each \( i = 1, \ldots, r \), define
	\[
	a_i = a_J\left( \frac{(f_1,\ldots,f_{i-1}) : f_i}{(f_1,\ldots,f_{i-1})} \right).
	\]
Set
$
T= \max_{2\le i\le r} \left\{ a_i-\sum_{j=i+1}^r a_j \right\} \quad \text{and} \quad T=0 \text{ if } r=1,
$
  $$N_2=\max\{a_1+\cdots+a_r+ T, \ar_J(f_1),\ldots,\ar_J(f_1,\ldots,f_r)\}+2.$$
	For every \( \varepsilon_1, \ldots, \varepsilon_r \in J^{N_2} \), let \( f_i' = f_i + \varepsilon_i \) for \( i = 1, \ldots, r \). Then, for each \( j \geq 0 \), we have:
	\begin{enumerate}\setlength{\itemsep}{0.5em}
		\item[\rm (i)] \( \beta^R_j\left(R / (f_1, \ldots, f_r)\right) = \beta^R_j\left(R / (f_1', \ldots, f_r')\right) \),
		\item[\rm (ii)] \( \mu^j_R\left(R / (f_1, \ldots, f_r)\right) = \mu^j_R\left(R / (f_1', \ldots, f_r')\right) \).
	\end{enumerate}
        \end{Theorem}

    \begin{proof}
The proof of \cite[Theorem 3.7]{VDT} works with Proposition 3.3 and Lemma 3.4 (valid for finitely
generated modules, with the same proofs) in place of the corresponding results of \cite{QNVT}; they give
$a_i'=a_i$ instead of $a_i'\le 2a_i$. Hence the bound of \cite{VDT} is replaced by $N_2=N+1$, where the
$+1$ comes from Proposition 4.1.
    \end{proof}



\section*{Acknowledgements} \sk
This work is supported by Vietnam National Program for the Development of Mathematics 2021-2030 under grant number B2027-CTT-02.
The author would like to thank Professor Pham Hung Quy for suggesting questions concerning explicit perturbation bounds.
The author also thanks the referees for their careful reading and helpful comments.

\end{document}